\documentclass[a4paper,leqno]{amsart}

\usepackage{amssymb,amsthm,amsmath,amsrefs}

\newcounter{mylisti} \newcounter{mylistii}
\newcounter{nest}
\newcommand{\defaultlabel}{}

\newcommand{\bn}{\ensuremath{\mathbb N}}

\newcommand{\cC}{\ensuremath{\mathcal C}}

\newcommand{\cP}{\ensuremath{\mathcal P}}

\newcommand{\cR}{\ensuremath{\mathcal R}}

\newcommand{\diam}{\operatorname{diam}}

\newcommand{\cdim}{\operatorname{cdim}}

\newcommand{\dist}{\ensuremath{\mathrm{dist}}}

\newcommand{\sep}{\ensuremath{\mathrm{sep}}}

\newcommand{\ds}{\displaystyle}

\newtheorem{theorem}{Theorem}

\newtheorem{lemma}[theorem]{Lemma}
\newtheorem{proposition}[theorem]{Proposition}
\newtheorem{corollary}[theorem]{Corollary}

\theoremstyle{definition}

\theoremstyle{remark}

\newcommand{\lip}{\ensuremath{\mathrm{Lip}}}

\newcommand{\co}{\mathrm{c}_0}

\newcommand{\xbl}{\ensuremath{\left(\sum_{n=1}^\infty\ell_{\infty}^n\right)_2}}
\begin{document}

\title[On the $(\beta)$-distortion of trees]{On the $(\beta)$-distortion of some infinite graphs}

\author{F.~Baudier}
\address{Institut de Math\'ematiques Jussieu-Paris Rive Gauche, Universit\'e Pierre et Marie Curie, Paris, France and Department of Mathematics, Texas A\&M University, College Station, TX 77843, USA}
\email{flo.baudier@imj-prg.fr}
\date{}

\thanks{}

\begin{abstract} 
We show a distortion lower bound of $\Omega(\log(h)^{1/p})$ when embedding the countably branching hyperbolic tree of height $h$ into a Banach space with an equivalent norm satisfying Rolewicz property $(\beta)$ with modulus of power type $p>1$. Similarly we show that a distortion lower bound of $\Omega(l^{1/p})$ is incurred when embedding the parasol graphs with $l$ levels into a Banach space with the above property. We discuss the optimality of our results as well as several applications.
\end{abstract}

\maketitle

\section{Introduction}
Let us first introduce a bit of notation and and a few definitions. The omitted definitions of notational convention from Banach space theory can be found in \cite{Handbook}. A weighted connected graph is a connected graph $G=(V,E)$ with a weight function $w\colon E\to[0,\infty)$. We say that the graph is unweighted is every edge has unit weight. $G$ will always be equipped with its canonical metric\footnote{$\rho_G$ is actually a semi-metric but we shall always assume that the weight function induces a metric.} $\rho_G$ on its set of vertices, where $\rho_G(x,y):=\inf\{\sum_{e\in P}w(e)\ |\textrm{P is a path connecting x to y}\}$. A weighted tree $T$ is an acyclic weighted connected graph. In a tree two vertices are connected by a unique path. When we root a tree at an arbitrary vertex $r$ the ancestor-descendant relationship between pairs of vertices is then well defined. We also define the height of a vertex $x$, denoted $h(x)$, as the distance of $X$ from the root. The height of a rooted tree is $h(T):=\sup_{x\in T}h(x)$. Note that for a rooted tree $\frac{\diam(T)}{2}\le h(T)\le \diam(T)$. The last common ancestor (in the ancestor-descendant relationship) of two vertices $x$ and $y$ is denoted by $lca(x,y)$. With this notation $\rho_T(x,y)=h(x)+h(y)-2h(lca(x,y))=\rho_T(x,lca(x,y))+\rho_T(lca(x,y),y).$ For a positive integer $h$, $T^\omega_h$ denotes the unweighted complete countably branching rooted tree of height $h$, while $T^\omega_\omega$ will be the unweighted complete countably branching rooted tree of infinite height.

\medskip

For any sequence $(x_n)_{n\ge1}$ in a Banach space $X$ let $\sep[(x_n)_{n\ge1}]:=\inf\{\|x_m-x_n\|\ |\ m\neq n; n,m\ge 1\}$ and say that $(x_n)_{n\ge 1}$ is $\epsilon$-separated if $\sep[(x_n)_{n\ge1}]\ge \epsilon$. According to Kutzarova \cite{Kutzarova1991} the norm of a Banach space satisfies Rolewicz property $(\beta)$ if and only if for every $\epsilon>0$ there exists $\delta(\epsilon)>0$ such that for every $x\in B_X$ and every $\epsilon$-separated sequence $(y_n)_{n\ge 1}\in B_X$ there exists some $n_0\in\bn$ such that $\|\frac{x+y_{n_0}}{2}\|\le 1-\delta(\epsilon)$. It is convenient to introduce what is called the $(\beta)$-modulus as follows:

$$ \overline{\beta}_X(t):=1-\sup\left\{\inf_{n\ge 1}\left\{\frac{\|x+y_n\|}{2}\right\}\ |\ (y_n)_{n\ge 1}\in B_X; \sep[(y_n)_{n\ge1}]\ge t; x\in B_X \right\}.$$ 
The norm of a Banach space satisfies Rolewicz property $(\beta)$ if and only if $\overline{\beta}_X(t)>0$ for every $t>0$.
When $\overline{\beta}_X(t)\ge ct^p$ for some universal constant $c>0$ and some exponent $p\in(1,\infty)$ one says that the norm satisfies Rolewicz property $(\beta)$ with power type $p$, or equivalently that the norm has a $(\beta)$-modulus of power type $p$.

\medskip

The {\it distortion} of a bi-Lipschitz embedding $f\colon (X,d_X) \to (Y,d_Y)$ is defined as $$ \dist(f):= \lip(f)\lip(f^{-1})=\sup_{x\neq y \in X}\frac{d_Y(f(x),f(y))}{d_X(x,y)}.\sup_{x\neq y \in X}\frac{d_X(x,y)}{d_Y(f(x),f(y))}.$$
As usual $c_{Y}(X):=\inf\{\dist(f)\ |\ f\colon X\to Y \textrm{ injective}\}$ denotes the $Y$-distortion of $X$. Note that for a finite metric space $X$ its $\ell_p$-distortion coincides with its $L_p$-distortion and we simply write $c_p(X)$. 

\medskip

In this article we are concerned with the quantitative embeddability of some infinite graphs into Banach spaces satisfying Rolewicz property $(\beta)$. More precisely let $p\in(1,\infty)$, and define 
$$\cC_{(\beta_p)}:=\{\textrm{Y separable with an equivalent norm with $(\beta)$-modulus of power type $p$}\},$$
and
$$\cC_{(\beta)}:=\{\textrm{Y is separable with an equivalent norm with property $(\beta)$}\}.$$
A typical example of a Banach space in $\cC_{(\beta_p)}$ is any $\ell_p$-sum of finite dimensional Banach spaces, in particular $\ell_p$. Since property $(\beta)$ implies reflexivity neither $\ell_1$ nor $\co$ are in $\cC_{(\beta)}$. In \cite{DKLR2014} it was shown that $X$ admits an equivalent norm with property $(\beta)$ if and only if $X$ admits an equivalent norm with property $(\beta)$ with modulus of power type $p$ for some $p\in(1,\infty)$. In other words, $\bigcup_{p\in(1,\infty)}\cC_{(\beta_p)}=\cC_{(\beta)}$. We study the $(\beta)$-distortion (resp. $(\beta_p)$-distortion) of metric spaces, namely the value of the parameter $c_{(\beta)}(X):=\inf\{c_{Y}(X)\ |\ Y\in \cC_{(\beta)}\}$ (resp. $c_{(\beta_p)}(X):=\inf\{c_{Y}(X)\ |\ Y\in \cC_{(\beta_p)}\}$). The parameter $c_{(\beta)}(X)$ is a measure of the best possible embedding of $X$ into a space with property $(\beta)$. The problem of estimating the $(\beta)$-distortion becomes interesting for non-locally finite (hence infinite) metric spaces since it follows from \cite{BaudierLancien2008} that $c_{(\beta)}(M)\le c$ for every locally finite metric space $M$, $c$ being some universal constant\footnote{$c$ is less than 181 (see \cite{Baudier2012})}.  

\medskip

In Section 2.1 we will prove that if $Y\in\cC_{(\beta_p)}$ then $c_{Y}(T^\omega_h)=\Omega(\log(h)^{1/p})$, i.e. $c_{Y}(T^\omega_h)\gtrsim \log(h)^{1/p}$ where as usual the symbol $\gtrsim$ is meant to hide a constant depending eventually on the geometry of the receiving space $Y$ but not on $h$. The proof is a combination of an asymptotic version of the prong bending lemma from \cite{Kloeckner2014} (see also \cite{Matousek1999} for a similar argument), and of a self-improvement argument of Johnson and Schechtman \cite{JohnsonSchechtman2009} which was elegantly implemented in the case of binary trees by Kloeckner \cite{Kloeckner2014}. The fact that this bound is tight when $Y=\ell_p$, i.e. $c_{\ell_p}(T^\omega_h)=O(\log(h)^{1/p})$, is explained in Section 2.3. Section 2.2 is dedicated to the case of the parasol graphs introduced by Dilworth, Kutzarova and Randrianarivony in \cite{DKR2014}. Section 3 gathers some applications and remarks regarding the asymptotic Ribe program, and the finite determinacy of bi-Lipschitz embeddability problems.

\section{Embeddability into spaces with Rolewicz property $(\beta)$}

\subsection{Complete countably branching trees}
We denote by $K_{\omega,1}$ the star graph with countably many branches, i.e. the bipartite graph that has a partition into exactly $2$ classes, one consisting of a singleton called the center, the other one consisting of countably many vertices called the leaves. In the sequel $b$ will denote the center. We choose an arbitrary leaf that we will denote by $r$, and fix a labeling $(t_i)_{i\ge 1}$ of the (countably many) remaining leaves. We this labeling in mind $K_{\omega,1}$ can be seen as an umbel with countably many pedicels, where $r$ stands for root, $b$ for the branching point on the stem, and $(t_i)_{i\ge 1}$ is a labeling of the tips of the pedicels. As usual $K_{\omega,1}$ is equipped with the shortest path metric. Roughly speaking the next lemma says that if anubel is embedded into a space with property $(\beta)$ then at least one pedicel has to bend towards the root and the distance from its tip to the root is shorter than expected. It can be seen as an asymptotic analogue of Lemma $2$ in \cite{Kloeckner2014}.

\begin{lemma}[Umbel pedicel bending lemma]\label{umbel} Let $Y$ be a Banach space whose norm satisfies Rolewicz property $(\beta)$ with power type $p$ $(p>1)$, then there exists $\gamma:=\gamma(Y)>0$ such that for every non-contractive map $f\colon K_{\omega,1}\to Y$ there exists some $i_0\in\bn$ such that $$\|f(r)-f(t_{i_0})\|\le 2\left(\lip(f)-\frac{\gamma}{\lip(f)^{p-1}}\right).$$
\end{lemma}

\begin{proof}
One may assume after an appropriate translation that $f(r)=0$. Fix some $\eta\in(0,\infty)$ to be chosen later and assume also that for every $i\in\bn$ one has $\|f(t_i)\|\ge2(\lip(f)-\eta)$. Then 

\begin{align*}
\|f(b)\|=\|f(b)-f(t_i)+f(t_i)\|&\ge 2(\lip(f)-\eta)-\lip(f)\\
                                             &\ge\lip(f)-2\eta,
\end{align*}
and 
\begin{align*}
\|f(b)-f(t_i)\| &\ge \|f(t_i)\|-\|f(b)\|\ge \|f(t_i)\|-\|f(b)-f(r)\|\\
                     &\ge 2(\lip(f)-\eta)-\lip(f)=\lip(f)-2\eta.
\end{align*}
Let $x=\ds\frac{f(b)}{\|f(b)\|}$ and $\ds p_i=\frac{f(t_i)-f(b)}{\|f(t_i)-f(b)\|}$. Clearly $\|x\|=\|p_i\|=1$.

\medskip

\noindent\textbf{Fact:}  \begin{align*}
\left\|\frac{x+p_i}{2}\right\|\ge 1-2\frac{\eta}{\lip(f)}.
\end{align*} 

\medskip

\noindent\textit{Proof of fact.}
Let $v_i=\|f(b)\|p_i$. Then 
\begin{align*}
\|f(b)+v_i-f(t_i)\|&=\left\|(f(t_i)-f(b))\frac{\|f(b)\|-\|f(t_i)-f(b)\|}{\|f(t_i)-f(b)\|}\right\|\\
                          &\le\big|\|f(b)\|-\|f(t_i)-f(b)\|\big|\\
                          &\le \lip(f)-(\lip(f)-2\eta)=2\eta,
\end{align*}

and 

\begin{align*}
\|f(b)+v_i\|&=\|f(b)-f(t_i)+f(t_i)+v_i\|\ge \|f(t_i)\|-\|f(b)-f(t_i)+v_i\|\\
                 &\ge 2(\lip(f)-\eta)-2\eta=2\lip(f)-4\eta.
\end{align*}
Therefore,
\begin{align*}
\left\|\frac{x+p_i}{2}\right\|&=\frac{\|f(b)+v_i\|}{2\|f(b)\|}\ge \frac{2\lip(f)-4\eta}{2\lip(f)}\\
                          &\ge 1-2\frac{\eta}{\lip(f)}.
\end{align*} 
Since the norm of $Y$ satisfies Rolewicz property $(\beta)$ with power type $p$ it implies that $\ds\sep[(p_i)_{i\ge 1}]<\left(\frac{2\eta}{c\lip(f)}\right)^{\frac{1}{p}}$ for some constant $c:=c(Y,p)\in(0,\infty)$, and hence there exist $n\neq m$ such that $\ds\|p_n-p_m\|<\left(\frac{2\eta}{c\lip(f)}\right)^{\frac{1}{p}}$. In particular $$\|v_n-v_m\|<\lip(f)\left(\frac{2\eta}{c\lip(f)}\right)^{\frac{1}{p}}.$$

But 
\begin{align*}
\|f(t_n)-f(t_m)\|&=\|f(t_n)-(f(b)+v_n)+v_n-v_m+(f(b)+v_m)-f(t_m)\|\\
                      &\le \|f(t_n)-(f(b)+v_n)\|+\|v_n-v_m\|+\|(f(b)+v_m)-f(t_m)\|\\
                      &\le 2\eta+\lip(f)\left(\frac{2\eta}{c\lip(f)}\right)^{\frac{1}{p}}+2\eta\\
                      &\le 4\eta+\lip(f)\left(\frac{2\eta}{c\lip(f)}\right)^{\frac{1}{p}}.
\end{align*}
Since $f$ is non-contracting one has $\ds2<4\eta+\lip(f)\left(\frac{2\eta}{c\lip(f)}\right)^{\frac{1}{p}}$ which is a contradiction for $\ds\eta=\frac{\gamma}{\lip(f)^{p-1}}$ for $\gamma$ small enough, indeed 
\begin{align*}
4\eta+\lip(f)\left(\frac{2\eta}{c\lip(f)}\right)^{\frac{1}{p}}&=4\frac{\gamma}{\lip(f)^{p-1}}+\lip(f)\left(\frac{2\gamma}{c\lip(f)\lip(f)^{p-1}}\right)^{\frac{1}{p}}\\
 &\le 4\gamma+\left(\frac{2\gamma}{c}\right)^{\frac{1}{p}},
\end{align*}
which can be made arbitrarily small.
\end{proof}

Recall that for a positive integer $h$, $T^\omega_h$ denotes the unweighted complete countably branching rooted tree of height $h$. We prove Theorem \ref{treedist} using a self-improvement argument \`a la Johnson and Schechtman.

\begin{theorem}\label{treedist} Let $Y$ be a Banach space admitting an equivalent norm with $(\beta)$-modulus of power type $p>1$, then $c_Y(T^\omega_h)=\Omega(\log(h)^{1/p})$.
\end{theorem}

\begin{proof}
Assume as we may that $f\colon T^\omega_h\to Y$ is a non-contractive bi-Lipschitz embedding. Let $k\in \bn$ such that $2^k\le h<2^{k+1}$. The first two levels of the tree $T^\omega_{2^k}$ can be seen as countably many umbels attached to a root. According to Lemma \ref{umbel}, in each of these umbels one can select a tip of a pedicel from the second level of $T^\omega_{2^k}$. Countably many umbels are attached to each selected tips of level $2$. In each of these umbels select an element at the level $4$ of $T^\omega_{2^k}$ according to Lemma \ref{umbel} and repeat this procedure until countably many leaves of $T^\omega_{2^k}$ have been selected. The set of selected vertices endowed with the induced metric is clearly isometric to $T^\omega_{2^{k-1}}$ (up to a scaling factor of $2$) and it is easy to see that an appropriate rescaling of $f$ induces a non-contracting embedding $f_1$ of $T^\omega_{2^{k-1}}$ into $Y$ with Lipschitz constant at most $\ds\lip(f)-\frac{\gamma}{\lip(f)^{p-1}}$. Repeating this procedure $k$ times we get a bi-Lipschitz embedding $f_k$ of $T^\omega_1$ into $Y$ such that $1\le \lip(f_k)\le\lip(f)-k\frac{\gamma}{\lip(f)^{p-1}}$. Therefore one has that $\lip(f)\gtrsim k^{1/p}$ and it follows easily that $\dist(f)\gtrsim \log(h)^{1/p}$. 
\end{proof}

\subsection{Parasol graphs}
In this section we consider one more time the graph $K_{\omega,1}$ define in the previous section in the umbel configuration. However we introduce an extra vertex, denoted $s$, and we attach $s$ to the tips of the pedicels. $P^\omega_1$ denotes the new graph obtained, which looks like a parasol. The proof of Lemma \ref{parasol} is very similar to the proof of Lemma \ref{umbel}.

\begin{lemma}[Parasol top bending lemma]\label{parasol} Let $Y$ be a Banach space whose norm satisfies Rolewicz property $(\beta)$ with power type $p$ $(p>1)$, then there exists $\gamma:=\gamma(Y)>0$ such that for every non-contractive map $f\colon P^\omega_1\to Y$ one has $$\|f(r)-f(s)\|\le 3\left(\lip(f)-\frac{\gamma}{\lip(f)^{p-1}}\right).$$
\end{lemma}

\begin{proof}
One may again assume after an appropriate translation that $f(r)=0$. Fix some $\eta\in(0,\infty)$ to be chosen later and assume also that $\|f(s)\|\ge3(\lip(f)-\eta)$. Then 
\begin{align*}
\|f(t_i)\|=\|f(t_i)-f(t)+f(t)\|&\ge 3(\lip(f)-\eta)-\lip(f)\\
                                             &\ge2\lip(f)-3\eta. 
\end{align*}
The condition above on the norm of the images of the tips is similar to the one at the beginning of Lemma \ref{umbel} with $\frac{3\eta}{2}$ instead of $\eta$. The extra edges added do not change the metric structure of the part of the graph which is in the umbel configuration, and the proof can be completed in the same fashion.  
\end{proof}

We now define as in \cite{DKR2014} the parasol graph $P^{\omega}_{l}$ using a fractal-like. The parasol graph of level $1$ is nothing else but the graph $P^{\omega}_{1}$. The parasol graph of level $2$ is simply the graph obtained by replacing each edge in $P^{\omega}_{1}$ with a copy of $P^{\omega}_{1}$. Proceeding recursively we construct the parasol graph of level $l$, namely, the parasol graph $P^{\omega}_{l}$ is obtained by replacing each edge in $P^{\omega}_{l-1}$ by a copy of $P^{\omega}_{1}$. A self-improvement argument can be applied to prove Theorem \ref{parasoldist}. Indeed if we simply select the root vertex and the summit vertex in each copy of $P^{\omega}_{1}$ constituting $P^{\omega}_{l}$ one obtains a rescaled isometric copy (scaling factor of $3$) of $P^{\omega}_{l-1}$ and we iterate this process $l$ times. 
\begin{theorem}\label{parasoldist} Let $Y$ be a Banach space admitting an equivalent norm with $(\beta)$-modulus of power type $p>1$, then $c_Y(P^{\omega}_{l})=\Omega(l^{1/p})$.
\end{theorem}

 

\subsection{Optimality of the results}
In this section we will show the optimality of Theorem \ref{treedist}.

\begin{proposition}\label{lpdistortiontree} Let $p\in(1,\infty)$. Let $T=(V,E)$ be an unweighted tree, then $c_{\ell_p(E)}(T)=O(\log(\diam T)^{\frac{1}{p}})$.
\end{proposition}
 
In particular since the edge-set of $T^\omega_h$ is countable one has that $c_{\ell_p}(T^\omega_h)=O(\log(h)^{\frac{1}{p}})$.  Since the norm of $\ell_p$ has $(\beta)$-modulus of power type $p$ the optimality of Theorem \ref{treedist} follows. Another consequence is the tight estimate $c_{\ell_p}(T^\omega_h)=\Theta(\log(h)^{\frac{1}{p}})$. Proposition \ref{lpdistortiontree} is probably well-known to the experts but since we could not locate an explicit proof we will provide one below for completeness and future reference. We also take this opportunity to discuss a beautiful but delicate argument of Matou\v{s}ek in \cite{Matousek1999}. Bourgain \cite{Bourgain1986a} gave an embedding 
of the complete binary tree of height $h$ into $\ell_2$ with distortion $O(\sqrt{\log(h)})$. Its proof extends verbatim to give a distortion upper bound of $O(\sqrt{\log(\diam(T))})$ for every (not necessarily finite) unweighted tree $T$. We will provide shortly the modifications (of the embedding and of the proof) required to achieve an embedding into an $\ell_p$-space. When dealing with weighted trees the situation becomes significantly more complicated (even for finite trees). Also it is clear that an upper-bound involving the diameter is not optimal (since a infinite path is a tree with infinite diameter), but sufficient to handle our examples. These two issues were taken care of in \cite{LinialMagenSaks1998} and \cite{Matousek1999}. Part of both arguments relies (implicitly in \cite{LinialMagenSaks1998} and explicitly in \cite{Matousek1999}) on the notion of caterpillar dimension of a finite tree, denoted $\cdim(T)$. Linial, Magen, and Saks showed that the Euclidean distortion of every \textit{finite weighted tree} $T$ with $l(T)$ leaves is bounded above by   $O(\log\log(l(T))$. Matou\v{s}ek proved that $c_{\ell_p}(T)=O(\log(\cdim(T))^{\min\{\frac{1}{2},\frac{1}{p}\}})$ for $p\in(1,\infty)$. It is also mentioned in \cite{Matousek1999} that the latter result holds for infinite trees since one can define a caterpillar dimension for infinite tree as well. In regards of Theorem \ref{treedist} there can be some confusion because as we will see $\cdim(T^\omega_h)=h$ for the natural extension of the caterpillar dimension to the infinite setting. This point will be clarified by the end of this section.

\medskip

Following the neat exposition in \cite{Gupta2000}, we define the caterpillar dimension of a tree which coincides with the classical notion already defined on the class of finite trees. Let $T$ be a weighted tree. We root $T$ at an arbitrary vertex $r\in V$. Recall that a leaf in a tree is a vertex of degree $1$. To extend the notion of caterpillar decomposition to the infinite setting we need to consider more general paths than just root-leaf paths. A root-leafend path is either a path from the root to a leaf or a ray starting at the root. A monotone path is then a path which is a subset of some root-leafend path. A caterpillar decomposition of $T$ is a partition of the edge-set $E$ of $T$ consisting only of monotone paths. The width of a caterpillar decomposition $\cP$, denoted by $\textrm{width}(\cP)$, is the smallest integer $m$ such that any root-leafend path in $T$ has a non-empty intersection with at most $m$ elements of $\cP$. The caterpillar dimension of $T$ is then $\cdim(T):=\inf\{\textrm{width}(\cP)\ |\ \cP \textrm{ is a caterpillar decomposition of } $T$ \}$. Since a finite tree does not have rays, and every root-leafend path is actually a root-leaf path, the definition coincides with the definition in the finitary setting. By induction it is fairly easy to show that for every tree finite tree $T$ one has $\cdim(T)=O(\log(l(T)))$ (c.f. \cite{LinialMagenSaks1998} or \cite{Matousek1999}). This fact does not hold in the infinite setting (think about the infinite complete rooted binary tree with one extra vertex attached to the root). It is easy to see that $\cdim(T^\omega_h))=h$. Indeed, if one takes the trivial caterpillar decomposition, i.e. $\cP=E$, then every root-leafend path intersects with $h$ elements of the decomposition. 

\medskip

The following theorem was proved in \cite{Matousek1999} in the finite setting, even though it was not explicitly stated in this form, and is readily extendable to the infinite setting with the definition of the caterpillar dimension given above. One can then argue using classical arguments, which can fail for non locally finite spaces, that the upper bound $c_{\ell_p}(T)=O(\log(\cdim(T))^{\min\{\frac{1}{2},\frac{1}{p}\}})$ for finite metric spaces follows from Theorem \ref{cdim}. 
\begin{theorem}\label{cdim}[Matou\v{s}ek] For any $p\in(1,\infty)$ and for any weighted tree $T$ there exists a set $I$ and an embedding of $T$ into $\ell_p(I)$ with distortion $O(\log(\cdim(T))^{\frac{1}{p}})$. 
\end{theorem}
We briefly sketch the proof of Theorem \ref{cdim}. Assume that $\cdim(T)=m$. Then $T$ admits some caterpillar decomposition $\cP$ with width less than $m$. We classify the monotone paths of the caterpillar decomposition as follows. The monotone paths whose oldest vertex (in the ancestor-descendant relationship) is the root are called the monotone paths of level $1$. The monotone paths of level $l$ are the monotone paths whose oldest vertex is a vertex of a monotone path of level $l-1$. By definition of the width of the caterpillar decomposition every monotone path of the partition belongs to exactly one level between level $1$ and level $m$. We now describe Matou\v{s}ek's embedding. Let $q$ be the conjugate exponent of $p\in(1,\infty)$, i.e. $\frac{1}{p}+\frac{1}{q}=1$. Denote by $(u_{P})_{P\in\cP}$ the canonical basis in $\ell_p(\cP)$. Let $x$ be some vertex in $T$. Denote by $P_k^x$ the monotone path of level $k$ that intersects the path from the root to $x$. There are $m(x)\le m$ such monotone paths. $l_k^x$ will denote the length of the portion of the path from the root to $x$ that belongs to the path $P_k^x$. For a real number $\alpha$ we will use the convenient notation $\alpha^+:=\max\{0;\alpha\}$. Finally the embedding $f$ from $T$ into $\ell_p(\cP)$
is given by $f(x)=\sum_{k=1}^{m(x)}\beta_k^x u_{P_k^x}$ where $\beta_{m(x)}^x:=l_{m(x)}^x$ and for $1\le k\le m(x)-1$, 

$$\beta_k^x:=l_k^x\left(1+\sum_{i=k+1}^{m(x)}\left(\frac{l_i^x}{l_k^x}-\frac{1}{2m}\right)^+\right)^{\frac{1}{q}}.$$

When $l_k^x=0$ we take $\beta_k^x=0$. It can then be checked that the tedious analysis of the distortion of the embedding can be carried out verbatim in the infinite setting. Instead of repeating the proof word by word we will prove Proposition \ref{lpdistortiontree} applying Matou\v{s}ek's approach in this particular case and show that it essentially boils down to an analogue of Bourgain's embedding when one considers the trivial caterpillar decomposition. The proof retains most of the arguments from the full proof of Theorem \ref{cdim} and we will discuss the missing arguments needed to reach the conclusion of Theorem \ref{cdim}.

\medskip

\noindent\textit{Proof of Proposition \ref{lpdistortiontree}.} Let $T$ be an unweighted rooted tree with finite height $h$ and consider the trivial caterpillar decomposition $\cP=E$. $\cP$ has width at most $h$. In this case one can take $m=h$ and $m(x)=h(x)$ the height of $x$. The paths of level $l$ are exactly the edges with one vertex at height $l-1$ and the other one at height $l$. Since $w(e)=1$ for every $e\in E$ one has $l_k^x=1$ for every $x\in T$ and every $k\in\{1,\cdots, h(x)\}$. The coefficients in Matou\v{s}ek's embedding become $\beta_{h(x)}^x=1$ and $\beta_k^x=\left(1+(h(x)-k)\frac{2h-1}{2h}\right)^{\frac{1}{q}}$ for $k\in\{1,\cdots, h(x)-1\}$. For $p=q=2$ this embedding is nothing else but Bourgain's embedding up to the coefficient $\frac{2h-1}{2h}$. We could keep this coefficient in the rest of the proof but we will end up realizing that the distortion obtained is the same up to some constant. We will just drop this coefficient which is a reminiscence of Matou\v{s}ek's embedding and that is not needed for the distortion we are aiming at. Therefore our embedding of $T$ into $\ell_p(E)$ is 

$$f(x)=\sum_{k=1}^{h(x)}\left(1+(h(x)-k)\right)^{\frac{1}{q}} u_{e^x_k},$$

where $e^x_k$ is the edge of level $k$ in the unique path from the root to $x$.
We now proceed with the analysis of the embedding. Let $x,y\in T$.
\begin{align*}
\|f(x)-f(y)\|_p^p&=\|\sum_{k=1}^{h(x)}\left(1+(h(x)-k)\right)^{\frac{1}{q}} u_{e^x_k}-\sum_{k=1}^{h(y)}\left(1+(h(y)-k)\right)^{\frac{1}{q}} u_{e^y_k}\|_p^p\\
                  &=\sum_{k=1}^{h(lca(x,y))}|\left(1+(h(x)-k)\right)^{\frac{1}{q}}-\left(1+(h(y)-k)\right)^{\frac{1}{q}}|^{p}\hskip 5mm (C)\\
                  &+\sum_{k=h(lca(x,y))+1}^{h(x)}\left(1+(h(x)-k)\right)^{p-1}\hskip 5mm (A)\\
                  &+\sum_{k=h(lca(x,y))+1}^{h(y)}\left(1+(h(y)-k)\right)^{p-1}\hskip 5mm (B)\\
\end{align*}
The quantities $A$ and $B$ are easy to estimate and we record it in the following fact.
\noindent\textbf{Fact 1:}  The following inequalities hold:

$$\left(\frac{h(x)-h(lca(x,y))}{2}\right)^{p}\le A\le \left(h(x)-h(lca(x,y))\right)^{p},$$
and
$$\left(\frac{h(y)-h(lca(x,y))}{2}\right)^{p}\le B\le \left(h(y)-h(lca(x,y))\right)^{p}.$$

\noindent\textit{Proof of Fact 1.}
\begin{align*}
A\le & \sum_{k=h(lca(x,y))+1}^{h(x)}\left(h(x)-h(lca(x,y))\right)^{p-1}\le \left(h(x)-h(lca(x,y))\right)^{p}.
\end{align*}

Denote by $j$ the unique integer in $[h(lca(x,y))+1,h(x)]$ such that $1+h(x)-k\ge \frac{h(x)-h(lca(x,y))}{2}$ for $k\le j$ and
$1+h(x)-(j+1)<\frac{h(x)-h(lca(x,y))}{2}$. 
\begin{align*}
A\ge & \sum_{k=h(lca(x,y))+1}^{j}\left(1+h(x)-k\right)^{p-1}\ge \sum_{k=h(lca(x,y))+1}^{j}\left(\frac{h(x)-h(lca(x,y))}{2}\right)^{p-1}\\
  \ge & (j-h(lca(x,y)))\left(\frac{h(x)-h(lca(x,y))}{2}\right)^{p-1}\ge \left(\frac{h(x)-h(lca(x,y))}{2}\right)^{p}
\end{align*}
By exchanging the role of $x$ and $y$ we get the same estimates for $B$.

\medskip

\noindent\textit{Compression of the embedding:}
\begin{align*}
\|f(x)-f(y)\|_p&\ge (A+B)^{\frac{1}{p}}\\
                  &\ge \left(\left(\frac{h(x)-h(lca(x,y))}{2}\right)^{p}+\left(\frac{h(y)-h(lca(x,y))}{2}\right)^{p}\right)^{\frac{1}{p}}\\
                  & \ge 2^{-\frac{1}{q}}\frac{h(x)-h(lca(x,y))}{2}+\frac{h(y)-h(lca(x,y))}{2}\\
                  &\ge 2^{\frac{1}{p}-2}\rho(x,y).
\end{align*}

\noindent\textit{Expansion of the embedding:}
To prove Fact 2, we will need the inequality which says that for every $s\in[0,1]$ and for every $a>b>0$ then $a^s-b^s\le\frac{a-b}{a^{1-s}}$. 

\medskip

\noindent\textbf{Fact 2:} For all $x,y\in T$, $C\le 2\rho(x,y)^p\log(\diam(T)).$

\medskip

\noindent\textit{Proof of Fact 2.} Since $T$ is a hyperbolic tree it is sufficient to prove Fact 2 in the situation were $lca(x,y)=x$. 
\begin{align*}
C&=\sum_{k=1}^{h(lca(x,y))}|\left(1+(h(x)-k)\right)^{\frac{1}{q}}-\left(1+(h(y)-k)\right)^{\frac{1}{q}}|^{p}\\
  &\le \sum_{k=1}^{h(lca(x,y))}\left(\frac{h(y)-h(x)}{(1+h(y)-k)^{\frac{1}{p}}}\right)^{p}=\sum_{k=1}^{h(x)}\frac{\rho(x,y)^p}{(1+h(y)-k)}\\
  &\le \rho(x,y)^p\sum_{k=2}^{h(y)}\frac{1}{k}\le 2\rho(x,y)^p\log(h(y))
\end{align*}
Combining all the estimates and rescaling by $2^{1/p}$ we obtain an embedding (still denoted $f$) from $T$ into $\ell_p(E)$ such that for every $x,y\in T$,
$$\frac{\rho(x,y)}{4}\le\|f(x)-f(y)\|_p\le 2^{\frac{1}{p}}\rho(x,y)\log(\diam(T))^{\frac{1}{p}}.$$
\hskip 12cm$\square$

\medskip

Considering weighted trees introduce some supplementary difficulties. For instance we have to be a little bit more careful at the branching point but this technicality is easily overcome and is not fundamental. However if you consider a ray attached to the root of the tree $T^\omega_h$ for instance, the new tree has infinite diameter and the choice of the trivial caterpillar decomposition is useless. So you have to include the ray as one of the monotone paths of a caterpillar decomposition of $T$. Doing this $\left(\frac{l_i^x}{l_k^x}-\frac{1}{2m}\right)^+$ which is never $0$ in our proof might eventually vanish. Sinc the $\log(\diam(T))$ term comes from a crude estimate in the last sum above, the fact that $\left(\frac{l_i^x}{l_k^x}-\frac{1}{2m}\right)^+$ might be $0$ for weighted trees with caterpillar decomposition less than $m$ reduces the contribution of the terms of the sum. That is the crucial point. 
\section{Applications}

\subsection{Asymptotic Ribe Program}
Recall briefly the asymptotic versions of uniform convexity and uniform
smoothness. Let $(X,\|\ \|)$ be a Banach space and $\tau>0$. We denote by $B_X$ its closed unit ball and by $S_X$ its unit sphere. For $x\in S_X$ and $Y$ a closed linear subspace of $X$, we define $$\overline{\rho}(\tau,x,Y)=\sup_{y\in S_Y}\|x+\tau y\|-1\ \ \ \ {\rm and}\ \ \ \ \overline{\delta}(\tau,x,Y)=\inf_{y\in S_Y}\|x+\tau y\|-1.$$ Then
$$\overline{\rho}(\tau)=\sup_{x\in S_X}\ \inf_{{\rm dim}(X/Y)<\infty}\overline{\rho}(\tau,x,Y)\ \ \ \ {\rm and}\ \ \ \ \overline{\delta}(\tau)=\inf_{x\in S_X}\ \sup_{{\rm dim}(X/Y)<\infty}\overline{\delta}(\tau,x,Y).$$ The norm $\|\ \|$ is said to be {\it asymptotically uniformly smooth} (a.u.s. in short) if $$\lim_{\tau \to 0}\frac{\overline{\rho}(\tau)}{\tau}=0.$$ It is said to be {\it asymptotically uniformly convex} (a.u.c. in short) if $$\forall \tau>0\ \ \ \ \overline{\delta}(\tau)>0.$$ These moduli have been first introduced by Milman in \cite{Milman1971}.
$\cR$ will denote the class of reflexive Banach spaces and we define
$$\mathcal{AUC}:=\{Y\ |\ Y\textrm{ is separable and has an equivalent a.u.c. norm}\}$$ and
$$\mathcal{AUS}:=\{Y\ |\ Y\textrm{ is separable and has an equivalent a.u.s. norm}\}.$$ 

We recall the main result from \cite{BKL2010}.

\begin{theorem}[Baudier-Kalton-Lancien]\label{BKL} Let $X$ be a reflexive Banach space. The following assertions are equivalent:
\begin{enumerate}
\item $X$ is not a.u.s. renormable or $X$ is not a.u.c. renormable,
\item $\sup_{h\ge 1}c_X(T^\omega_h)<\infty$,
\item $c_X(T^\omega_\omega)<\infty$.
\end{enumerate}
\end{theorem}

Theorem \ref{equivalences} which was shown in \cite{DKLR2014}, says that the class $\cC_{(\beta)}$ and the class $\cR\cap\mathcal{AUC}\cap\mathcal{AUS}$ coincide. 
\begin{theorem}\label{equivalences} Let X be a separable reflexive Banach space. The following assertions are equivalent:
\begin{enumerate}
\item $X$ admits an equivalent norm with property $(\beta)$,
\item $X$ is a.u.s. renormable and a.u.c. renormable.
\end{enumerate}
\end{theorem} 

The same equivalences also hold without the separability assumption \cite{DKLRpriv}. Note that the equivalence with assertion $(2)$ is not explicit in \cite{DKLR2014} but follows from the proof of Theorem 4 in \cite{Kutzarova1990}. 

\begin{corollary}Let $X$ be a reflexive Banach space. The following assertions are equivalent:
\begin{enumerate}
\item $X$ admits an equivalent norm with property $(\beta)$,
\item $X$ is a.u.s. renormable and a.u.c. renormable,
\item $\sup_{h\ge 1}c_X(T^\omega_h)=\infty$,
\item $c_X(T^\omega_\omega)=\infty$.
\end{enumerate}
\end{corollary}

Using \cite{BKL2010} as a black box to prove that for every Banach space $Y$ admitting an equivalent norm with property one has $\lim_{h\to\infty}c_Y(T_h)=\infty$, does not give a good estimate on the rate of growth of $(c_Y(T_h))_{h\ge 1}$. Moreover the proof of the implication in \cite{BKL2010} stating that $\lim_{h\to\infty}c_Y(T_h)=\infty$ for every Banach space $Y$ that is a.u.s. renormable and a.u.c. renormable, is rather technical and escape the geometric intuition. Theorem \ref{treedist} gives a simple geometric direct proof of $(1)$ implies $(3)$ and provides an optimal estimate on the rate of growth. Moreover it can be applied to graphs with a completely different geometry.

\begin{corollary}Let $Y\in\cR\cap\mathcal{AUC}\cap\mathcal{AUS}$, then there exists $p\in(1,\infty)$ such that 
$$c_Y(P^{\omega}_{l})=\Omega(l^{1/p}).$$
In particular $c_Y(P^{\omega}_{\omega})=\infty.$
\end{corollary}

\subsection{Finite determinacy of bi-Lipschitz embeddability problems}

Let $\lambda\in[1,\infty)$. We say that a metric space $X$ is $\lambda$-finitely representable into another metric space $Y$ if for every \textit{finite} subset $F$ of $X$ one has $c_{Y}(F)\le \lambda$. We simply say that $X$ is crudely finitely representable if it is $\lambda$-finitely representable for some $\lambda\in[1,\infty)$.

\medskip

Let $\cC$ be a class of metric spaces. Given a metric space $X$, we say that its bi-Lipschitz embeddability problem for the class $\cC$ is finitely determined if for \textit{every} $Y\in \cC$, $X$ admits a bi-Lipschitz embedding into $Y$ whenever $X$ is crudely finitely representable in $Y$. Ostrovskii's finite determinacy theorem \cite{Ostrovskii2012} says that for every locally finite metric space $X$, its bi-Lipschitz embeddability problem for the class of Banach spaces is finitely determined. It is folklore that the locally finiteness condition in Ostrovskii's theorem can not be removed. For instance, $\ell_2$ is finitely representable into $\ell_1$ but it is now well known that $\ell_2$ does not bi-Lipschitz embed into $\ell_1$. If we restrict our attention to the class of graph metrics it becomes a non-trivial task to find examples of non locally finite graphs whose bi-Lipschitz embeddability problem for the class of Banach spaces is \textit{not} finitely determined. Appealing to Theorem \ref{treedist} (or Theorem \ref{BKL}) we can provide such an example. Indeed, if $Y=\xbl$, then $T_\omega^\omega$ is finitely representable in $Y$ but it does not admit any bi-Lipschitz embedding into $Y$. Therefore Ostrovskii's finite determinacy theorem does not hold even for structurally simple graphs such as (non locally finite) trees. Ostrovskii's proof actually gives a more precise quantitative statement.

\begin{theorem}[Ostrovskii] There exists $\gamma\in(0,\infty)$ such that for every locally finite metric space $M$ and every Banach space $Y$, the inequality $c_Y(M)\le \gamma\lambda$ holds whenever $M$ is $\lambda$-finitely representable in $Y$.
\end{theorem}

The bi-Lipschitz embeddability problem for the space $T_h^\omega$ seems to be more elusive. In the next proposition we show that an analogue of the quantitative statement above does not hold for the sequence $(T^\omega_h)_{h\ge 1}$. 

\begin{proposition}\label{uniformfindet}Let $p\in(1,2)$. There does not exists a constant $\gamma\in(0,\infty)$ such that for every $h\ge 1$, the inequality $c_{\ell_p}(T^\omega_h)\le \gamma\lambda$ holds whenever $T^\omega_h$ is $\lambda$-finitely representable in $\ell_p$.
\end{proposition}

\begin{proof} Fix $p\in(1,2)$ and assume that there exist a finite constant $\beta>0$ such that for every $h\ge 1$, the inequality $c_{\ell_p}(T^\omega_h)\le \beta\alpha$ holds whenever for every finite subset $F\subset T^\omega_h$ one has $c_{\ell_p}(F)\le \alpha$. Let $F$ be a finite subset of $T^\omega_h$. It follows from Proposition \ref{lpdistortiontree} that $c_2(F)\lesssim \sqrt{\log(h)}$. It is important to notice that the constant depends only on the Hilbert space but not on $h$. Since it is well known that every finite subset of a Hilbert space is isometric to a subset of $\ell_p$ one gets that $c_{\ell_p}(F)\lesssim \sqrt{\log(h)}$. But Theorem \ref{treedist} says that $c_{\ell_p}(T^\omega_h)\gtrsim \log(h)^{\frac{1}{p}}$ where the constant depends solely on $p$. Therefore $\beta\gtrsim\log(h)^{\frac{1}{p}-\frac{1}{2}}$ which is a contradiction for $p\in(1,2)$ and $h$ large enough.
\end{proof}

\bigskip

\textbf{Acknowledgments:} We would like to extend our gratitude and appreciation to Florence Lancien, Gilles Lancien, and Tony Proch\'azka for the flawless organization of the Autumn School on Nonlinear Geometry of Banach Spaces and Applications in M\'etabief and of the Conference on Geometric Functional Analysis and its Applications in Besan\c con. The scientific activity and atmosphere was incredibly enlightening. The work presented here found its inspiration and was carried out while participating at these events.

\end{document}